\documentclass[a4paper,11pt]{article}
\usepackage{amsmath,amssymb,amsfonts,mathrsfs,hyperref,microtype, amsthm}
\usepackage{eufrak}

\newtheorem{thm}{Theorem}[section]
\newtheorem{rmk}{Remark}[section]
\newtheorem{definition}{Definition}[section]
\newtheorem{lemma}{Lemma}[section]
\newtheorem{propo}{Proposition}[section]
\newtheorem{coro}{Corollary}[section]

\newtheorem{ex}{Example}[section]

\numberwithin{equation}{section}
\usepackage{graphicx,color}
\usepackage[T1]{fontenc}
\usepackage{rotating}
\usepackage{booktabs}
\usepackage{mathtools}
\usepackage{float}

\usepackage{breqn}
\usepackage{graphicx}
\usepackage{caption}

\def\HH{ \EuFrak H}
\def\N{{\rm I\kern-0.16em N}}
\def\R{{\rm I\kern-0.16em R}}
\def \E{{\rm I\kern-0.16em E}}
\def\P{{\rm I\kern-0.16em P}}
\def\F{{\rm I\kern-0.16em F}}
\def\B{{\rm I\kern-0.16em B}}
\def\C{{\rm I\kern-0.46em C}}
\def\G{{\rm I\kern-0.50em G}}

\newcommand{\K}{\mathrm{\textbf{Ker}}}
\newcommand{\Id}{\mathrm{\textbf{Id}}}
\newcommand{\LL}{\mathrm{L}}
\newcommand{\ud}{\mathrm{d}}

\numberwithin{equation}{section}
\font\eka=cmex10
\usepackage{ae}
\def\ind{\mathrel{\hbox{\rlap{%
\hbox to 7.5pt{\hrulefill}}\raise6.6pt\hbox{\eka\char'167}}}}
\begin{document}

\title{\textbf{Convergence towards linear combinations \\ of chi-squared random variables: \\ a Malliavin-based approach}}

\author{Ehsan Azmoodeh \\ \small{Unit\'{e} de Recherche en Math\'{e}matiques, Luxembourg University}\\ \medskip\\
Giovanni Peccati \\ \small{Unit\'{e} de Recherche en Math\'{e}matiques, Luxembourg University}\\ \smallskip\\
Guillaume Poly \\ \small{Institut de Recherche Math\'ematiques de Rennes}}



\maketitle

\center{ {\it Dedicated to the memory of Marc Yor}}

\abstract 

We investigate the problem of finding necessary and sufficient conditions for convergence in distribution towards a general finite linear combination of independent chi-squared random variables, within the framework of random objects living on a fixed Gaussian space. Using a recent representation of cumulants in terms of the Malliavin calculus operators $\Gamma_i$ (introduced by Nourdin and Peccati in \cite{n-pe-3}), we provide conditions that apply to random variables living in a finite sum of Wiener chaoses. As an important by-product of our analysis, we shall derive a new proof and a new interpretation of a recent finding by Nourdin and Poly \cite{n-po-1}, concerning the limiting behaviour of random variables living in a Wiener chaos of order two. Our analysis contributes to a fertile line of research, that originates from questions raised by Marc Yor, in the framework of limit theorems for non-linear functionals of Brownian local times.

\vskip0.3cm
\noindent {\bf Keywords}: Gaussian analysis, Malliavin calculus, multiple integrals, cumulants, noncentral limit theorems, weak convergence.

\noindent{\bf MSC 2010:}  60F05; 60G15; 60H07
\tableofcontents
\section{Introduction}

The aim of this paper is to provide necessary and sufficient conditions (expressed in terms of Malliavin operators), ensuring that a sequence of random variables living in a finite sum of Wiener chaoses converges in distribution towards a finite linear combination of independent centered chi-squared random variables. As discussed below, we regard the results of the present paper as a first step towards the solution of an open and notoriously difficult problem, namely: {\it can one derive necessary and sufficient analytical conditions, ensuring that a given sequence of smooth functionals of a Gaussian field converge in distribution towards an element of the second Wiener chaos?} Finite linear combinations of independent chi-squared random variables represent indeed the most elementary instance of random objects living in the second Wiener chaos of a Gaussian field (see Section \ref{ss:swc} below for a discussion of this point). More sophisticated examples -- that are crucial for applications and lay at present largely outside the scope of Malliavin-type techniques -- include the so-called {\it Rosenblatt distribution}; see e.g. \cite{tudor_book} for a detailed discussion of these objects. 

\subsection{Overview}

We refer the reader to \cite{n-pe-1}, as well as Section 2 below, for any unexplained notion evoked in the present section. Let $W = \{W(h) : \HH\}$ be an isonormal Gaussian process over some real separable Hilbert space $\HH$ and let $q\geq 1$. For every deterministic symmetric kernel $f \in \HH^{\odot q}$, we denote by $I_q(f)$ the multiple stochastic Wiener-It\^o integral of $f$ with respect to $W$.  Random variables of the form $I_q(f)$ compose the so-called $q$th {\it Wiener chaos} associated with $W$. The concept of Wiener chaos represents a rough infinite-dimensional analogous of the Hermite polynomials for the one-dimensional Gaussian distribution (see e.g. \cite{n-pe-1, PecTaq} for a detailed discussion of these objects). 

The following two results, proved respectively in \cite{NO, n-p-05} and \cite{n-pe-2}, contain an exhaustive characterization of normal and Gamma approximations on Wiener chaos. As in \cite{n-pe-2}, we denote by  $F(\nu)$ a centered random variable with the law of $2G(\nu/2)-\nu$, where
$G(\nu/2)$ has a Gamma distribution with parameter $\nu/2$. In particular, when $\nu\geq 1$ is an integer, then $F(\nu)$ has a centered $\chi^2$ distribution with $\nu$ degrees of freedom. 

\begin{thm}\label{T:NPEC} {\rm (A) (See \cite{NO, n-p-05})} Denote by $D$ the Malliavin derivative associated with $W$. Let $N\sim \mathcal{N}(0,1)$, fix $q\geq 2$ and let $I_q(f_n)$ be a sequence of multiple stochastic integrals with respect to $W$, with each $f_n$ a an element of $\HH^{\odot q}$ such that $E[I_q(f_n)^2] = 1$. Then, the following are equivalent, as $n\to \infty$: 
\begin{enumerate}
\item[ \rm (i) ] $I(f_n)$ converges in distribution to $N$;
\item[ \rm (ii) ] $E[I_q(f_n)^4]\to E[N^4] = 3$;
\item[\rm (iii)] $q^{-1} \|DI_q(f_n)\|^2_\HH  \to 1$ in $L^2(\Omega)$.
\end{enumerate}

\smallskip

\noindent {\rm (B) (See \cite{n-pe-2})} Fix $\nu >0$, and let $F(\nu)$ have the centered Gamma distribution described above. Let $q\geq 2$ be an even integer, and let $I_q(f_n)$ be a sequence of multiple integrals, with each $f_n\in \HH^{\odot q}$ verifying $E[I_q(f_n)^2]  =2\nu$. Then, the following are equivalent, as $n\to \infty$: 
\begin{enumerate}
\item[ \rm (i)] $I_q(f_n)$ converges in distribution to $F(\nu)$;
\item[ \rm (ii)] $E[I(f_n)^4] -12E[I_q(f_n)^3]  \to E[F(\nu)^4] -12E[F(\nu)^3]   = 12\nu^2 - 48\nu$;
\item[\rm (iii)] $\|DI_q(f_n)\|^2_\HH-2qI_q(f_n)-2q\nu \to 0$, in $L^2(\Omega)$.
\end{enumerate}
\end{thm}

The line of research associated with the content of Theorem \ref{T:NPEC} originates from some deep questions asked by Marc Yor, about the asymptotic behaviour of non-linear functionals of Brownian local times (partially addressed in references \cite{py1, py2}). As demonstrated e.g. in \cite{n-p-05}, results of this type are intimately connected to the powerful technique of {\it Brownian time changes} and associated limit theorems (a beautiful discussion of these topics can be found in \cite[Chapter V and Chapter XIII]{RY}): as such, they provide a drastic simplification of the so-called {\it method of moments} for probabilistic approximations. 

\smallskip

Theorem \ref{T:NPEC} has triggered a huge amount of applications and generalizations, involving e.g. Stein's method, stochastic geometry, free probability, power variations of Gaussian processes and analysis of isotropic fields o homogeneous spaces. See \cite{n-pe-2} for an introduction to this field of research. See \cite{WWW} for a constantly updated web resource, with links to all available papers.

\medskip

As anticipated, the aim of the present paper is to address the following question: {\it for a general $q$, is it possible to prove a statement similar to Part {\rm (B)} of Theorem \ref{T:NPEC}, when the the target distribution $F(\nu)$ is replaced by an object of the type
\begin{equation}\label{e:introtarget}
F_\infty = \sum_{i=1}^k \alpha_i(N_i^2-1),
\end{equation}
where $k$ is a finite integer, the $\alpha_i$, $i=1,...,k$, are pairwise distinct real numbers, and $\{N_i : i=1,...,k\}$ is a collection of i.i.d. $\mathcal{N}(0,1)$ random variables?} 

The following remarks are in order

\begin{itemize}
\item[--] In the case $q=2$ (that is, when the involved sequence of stochastic integrals belong to the second Wiener chaos of $W$), the question has been completely answered by Nourdin and Poly \cite{n-po-1}.

\item[--] The case $k=\alpha_1 = 1$ corresponds to Part (B) of Theorem \ref{T:NPEC}, in the special case $\nu=1$.

\item[--] When $k=2$ and $\alpha_1=\frac12 = -\alpha_2$, then one has that $F_\infty$ has the same law as the random variable $N_1\times N_2$. It is a well-known fact that the law of this random variable belongs to the general class of {\it Variance-Gamma} distributions: it follows that, in this special case, convergence towards $F_\infty$ could be studied by means of the general Malliavin-Stein techniques developed by Eichelsbacher and Th\"ale in the (independently written) paper \cite{et} (see also \cite{gaunt} for some related estimates). We observe that, in contrast to the present paper, the techniques developed in \cite{et} yield explicit rates of convergence in some probability metric. On the other hand, our approach allows one to deal with target probability distributions that fall outside the class of Variance-Gamma laws, as well as to deduce necessary conditions for the convergence to take place.
\end{itemize}

In order to deal with the previously stated problem, one cannot rely on techniques that have been used in the previous literature on related subjects. In particular:

\begin{itemize}
\item[(a)] For a general choice of $k$ and $\alpha_1, \cdots, \alpha_k$ there is no suitable version of Stein's method that can be applied to the random variable $F_\infty$ in \eqref{e:introtarget}, so that the Malliavin-Stein approach for normal and Gamma approximations developed in \cite{n-pe-ptrf} cannot be used. 

\item[(b)] For a general choice of $k$ and $\alpha_1, \cdots, \alpha_k$, it seems difficult to represent the characteristic function of $F_\infty$ as the solution of an ordinary differential equation: it follows that the characteristic function approach exploited in \cite{n-pe-2, NO} is not adapted to the framework of the present paper.

\item[(c)] The analytical approach used in \cite{n-po-1} (for the case $q=2$) cannot be applied in the case of a general order $q\geq 3$ since, in this case, the characteristic function of a non-zero random variable of the type $I_q(f)$ is not analytically known.

\end{itemize}

The main contribution of the present paper (stated in Theorem \ref{main-general-thm}) is a full generalisation of the double implication (iii) $\leftrightarrow$ (i) in the statement of  Theorem \ref{T:NPEC}-(B) to the case of a general target random variable of the form \eqref{e:introtarget} and of a general sequence of random variables living in a finite sum of Wiener chaoses. Our approach is based on a suitable extension of the method of moments, that relies in turn on several extensions of the results proved in \cite{n-po-1}. One should notice that our findings involve the operators $\Gamma_i$ from Malliavin calculus, as introduced in the reference \cite{n-pe-3} (see also \cite[Chapter 8]{n-pe-1}).

\begin{rmk}{\rm For the time being (and for technical reasons that will clearly appear in the sections to follow), it seems very arduous to extend the double implication (ii) $\leftrightarrow$ (i) in the statement of  Theorem \ref{T:NPEC}-(B).}
\end{rmk}

\subsection{Plan}

The paper is organized as follows. Section $2$ contains some preliminary materials including basic facts on Gaussian analysis and Malliavin calculus. Section $3$ is devoted to our main results on a general criterion for 
convergence in distribution towards chi-squared combinations, whereas Section 4 provides some examples.  

\section{Elements of Gaussian analysis and Malliavin calculus}\label{gaussian+malliavin}

This section contains the essential elements of Gaussian analysis and Malliavin calculus that are used in this paper. See for instance the references \cite{n-pe-1, nualart} for further details.

\subsection{Isonormal processes and multiple integrals}\label{ss:isonormal}

Let $ \EuFrak H$ be a real separable Hilbert space. For any $q\geq 1$, we write $ \EuFrak H^{\otimes q}$ and $ \EuFrak H^{\odot q}$ to
indicate, respectively, the $q$th tensor power and the $q$th symmetric tensor power of $ \HH$; we also set by convention
$ \EuFrak H^{\otimes 0} =  \HH^{\odot 0} =\R$. When $\HH = L^2(A,\mathcal{A}, \mu) =:L^2(\mu)$, where $\mu$ is a $\sigma$-finite
and non-atomic measure on the measurable space $(A,\mathcal{A})$, then $ \EuFrak H^{\otimes q} = L^2(A^q,\mathcal{A}^q,\mu^q)=:L^2(\mu^q)$, and $ \EuFrak H^{\odot q} = L_s^2(A^q,\mathcal{A}^q,\mu^q) := L_s^2(\mu^q)$, 
where $L_s^2(\mu^q)$ stands for the subspace of $L^2(\mu^q)$ composed of those functions that are $\mu^q$-almost everywhere symmetric. We denote by $W=\{W(h) : h\in  \EuFrak H\}$
an {\it isonormal Gaussian process} over $ \EuFrak H$. This means that $W$ is a centered Gaussian family, defined on some probability space $(\Omega ,\mathcal{F},P)$, with a covariance structure given by the relation
$E\left[ W(h)W(g)\right] =\langle h,g\rangle _{ \EuFrak H}$. We also assume that $\mathcal{F}=\sigma(W)$, that is, $\mathcal{F}$ is generated by $W$, and use the shorthand notation $L^2(\Omega) := L^2(\Omega, \mathcal{F}, P)$.

For every $q\geq 1$, the symbol $C_{q}$ stands for the $q$th {\it Wiener chaos} of $W$, defined as the closed linear subspace of $L^2(\Omega ,\mathcal{F},\P) =:L^2(\Omega) $
generated by the family $\{H_{q}(W(h)) : h\in  \EuFrak H,\left\| h\right\| _{ \EuFrak H}=1\}$, where $H_{q}$ is the $q$th Hermite polynomial, defined as follows:
\begin{equation}\label{hq}
H_q(x) = (-1)^q e^{\frac{x^2}{2}}
 \frac{d^q}{dx^q} \big( e^{-\frac{x^2}{2}} \big).
\end{equation}
We write by convention $C_{0} = \mathbb{R}$. For
any $q\geq 1$, the mapping $I_{q}(h^{\otimes q})=H_{q}(W(h))$ can be extended to a
linear isometry between the symmetric tensor product $ \EuFrak H^{\odot q}$
(equipped with the modified norm $\sqrt{q!}\left\| \cdot \right\| _{ \EuFrak H^{\otimes q}}$)
and the $q$th Wiener chaos $C_{q}$. For $q=0$, we write by convention $I_{0}(c)=c$, $c\in\mathbb{R}$.

It is well-known that $L^2(\Omega)$ can be decomposed into the infinite orthogonal sum of the spaces $C_{q}$: this means that any square-integrable random variable
$F\in L^2(\Omega)$ admits the following {\it Wiener-It\^{o} chaotic expansion}
\begin{equation}
F=\sum_{q=0}^{\infty }I_{q}(f_{q}),  \label{E}
\end{equation}
where the series converges in $L^2(\Omega)$, $f_{0}=E[F]$, and the kernels $f_{q}\in  \EuFrak H^{\odot q}$, $q\geq 1$, are
uniquely determined by $F$. For every $q\geq 0$, we denote by $J_{q}$ the
orthogonal projection operator on the $q$th Wiener chaos. In particular, if
$F\in L^2(\Omega)$ has the form (\ref{E}), then
$J_{q}F=I_{q}(f_{q})$ for every $q\geq 0$.

Let $\{e_{k},\,k\geq 1\}$ be a complete orthonormal system i $\HH$. Given $f\in  \EuFrak H^{\odot p}$ and $g\in \EuFrak H^{\odot q}$, for every
$r=0,\ldots ,p\wedge q$, the \textit{contraction} of $f$ and $g$ of order $r$
is the element of $ \EuFrak H^{\otimes (p+q-2r)}$ defined by
\begin{equation}
f\otimes _{r}g=\sum_{i_{1},\ldots ,i_{r}=1}^{\infty }\langle
f,e_{i_{1}}\otimes \ldots \otimes e_{i_{r}}\rangle _{ \EuFrak H^{\otimes
r}}\otimes \langle g,e_{i_{1}}\otimes \ldots \otimes e_{i_{r}}
\rangle_{ \EuFrak H^{\otimes r}}.  \label{v2}
\end{equation}
Notice that the definition of $f\otimes_r g$ does not depend
on the particular choice of $\{e_k,\,k\geq 1\}$, and that
$f\otimes _{r}g$ is not necessarily symmetric; we denote its
symmetrization by $f\widetilde{\otimes }_{r}g\in  \EuFrak H^{\odot (p+q-2r)}$.
Moreover, $f\otimes _{0}g=f\otimes g$ equals the tensor product of $f$ and
$g$ while, for $p=q$, $f\otimes _{q}g=\langle f,g\rangle _{ \EuFrak H^{\otimes q}}$. 
When $\HH = L^2(A,\mathcal{A},\mu)$ and $r=1,...,p\wedge q$, the contraction $f\otimes _{r}g$ is the element of $L^2(\mu^{p+q-2r})$ given by
\begin{eqnarray}\label{e:contraction}
&& f\otimes _{r}g (x_1,...,x_{p+q-2r})\\
&& = \int_{A^r} f(x_1,...,x_{p-r},a_1,...,a_r)\times \notag\\
&& \quad\quad\quad\quad \times g(x_{p-r+1},...,x_{p+q-2r},a_1,...,a_r)d\mu(a_1)...d\mu(a_r). \notag
\end{eqnarray}

It is a standard fact of Gaussian analysis that the following {\it multiplication formula} holds: if $f\in  \EuFrak
H^{\odot p}$ and $g\in  \EuFrak
H^{\odot q}$, then
\begin{eqnarray}\label{multiplication}
I_p(f) I_q(g) = \sum_{r=0}^{p \wedge q} r! {p \choose r}{ q \choose r} I_{p+q-2r} (f\widetilde{\otimes}_{r}g).
\end{eqnarray}
\smallskip

\subsection{Malliavin operators}\label{ss:mall}

We now introduce some basic elements of the Malliavin calculus with respect
to the isonormal Gaussian process $W$. Let $\mathcal{S}$
be the set of all
cylindrical random variables of
the form
\begin{equation}
F=g\left( W(\phi _{1}),\ldots ,W(\phi _{n})\right) ,  \label{v3}
\end{equation}
where $n\geq 1$, $g:\mathbb{R}^{n}\rightarrow \mathbb{R}$ is an infinitely
differentiable function such that its partial derivatives have polynomial growth, and $\phi _{i}\in  EuFrak H$,
$i=1,\ldots,n$.
The {\it Malliavin derivative}  of $F$ with respect to $W$ is the element of $L^2(\Omega , \EuFrak H)$ defined as
\begin{equation*}
DF\;=\;\sum_{i=1}^{n}\frac{\partial g}{\partial x_{i}}\left( W(\phi_{1}),\ldots ,W(\phi _{n})\right) \phi _{i}.
\end{equation*}
In particular, $DW(h)=h$ for every $h\in  \EuFrak H$. By iteration, one can define the $m$th derivative $D^{m}F$, which is an element of $L^2(\Omega , \EuFrak H^{\odot m})$,
for every $m\geq 2$.
For $m\geq 1$ and $p\geq 1$, ${\mathbb{D}}^{m,p}$ denotes the closure of
$\mathcal{S}$ with respect to the norm $\Vert \cdot \Vert _{m,p}$, defined by
the relation
\begin{equation*}
\Vert F\Vert _{m,p}^{p}\;=\;E\left[ |F|^{p}\right] +\sum_{i=1}^{m}E\left[
\Vert D^{i}F\Vert _{ EuFrak H^{\otimes i}}^{p}\right].
\end{equation*}
We often use the (canonical) notation $\mathbb{D}^{\infty} := \bigcap_{m\geq 1}
\bigcap_{p\geq 1}\mathbb{D}^{m,p}$.

\begin{rmk}\label{r:density}{\rm It is a well-known fact that any random variable $F$ that is a finite linear combination of multiple Wiener-It\^o integrals is an element of $\mathbb{D}^\infty$. }
\end{rmk}

The Malliavin derivative $D$ obeys the following \textsl{chain rule}. If
$\varphi :\mathbb{R}^{n}\rightarrow \mathbb{R}$ is continuously
differentiable with bounded partial derivatives and if $F=(F_{1},\ldots
,F_{n})$ is a vector of elements of ${\mathbb{D}}^{1,2}$, then $\varphi
(F)\in {\mathbb{D}}^{1,2}$ and
\begin{equation}\label{e:chainrule}
D\,\varphi (F)=\sum_{i=1}^{n}\frac{\partial \varphi }{\partial x_{i}}(F)DF_{i}.
\end{equation}

Note also that a random variable $F$ as in (\ref{E}) is in ${\mathbb{D}}^{1,2}$ if and only if
$\sum_{q=1}^{\infty }q\|J_qF\|^2_{L^2(\Omega)}<\infty$
and in this case one has the following explicit relation: $$E\left[ \Vert DF\Vert _{ \EuFrak H}^{2}\right]
=\sum_{q=1}^{\infty }q\|J_qF\|^2_{L^2(\Omega)}.$$ If $ \EuFrak H=
L^{2}(A,\mathcal{A},\mu )$ (with $\mu $ non-atomic), then the
derivative of a random variable $F$ as in (\ref{E}) can be identified with
the element of $L^2(A \times \Omega )$ given by
\begin{equation}
D_{t}F=\sum_{q=1}^{\infty }qI_{q-1}\left( f_{q}(\cdot ,t)\right) ,\quad t \in A.  \label{dtf}
\end{equation}


The operator $L$, defined as $L=\sum_{q=0}^{\infty }-qJ_{q}$, is the {\it infinitesimal generator of the Ornstein-Uhlenbeck semigroup}. The domain of $L$ is
\begin{equation*}
\mathrm{Dom}L=\{F\in L^2(\Omega ):\sum_{q=1}^{\infty }q^{2}\left\|
J_{q}F\right\| _{L^2(\Omega )}^{2}<\infty \}=\mathbb{D}^{2,2}\text{.}
\end{equation*}


For any $F \in L^2(\Omega )$, we define $L^{-1}F =\sum_{q=1}^{\infty }-\frac{1}{q} J_{q}(F)$. The operator $L^{-1}$ is called the
\textit{pseudo-inverse} of $L$. Indeed, for any $F \in L^2(\Omega )$, we have that $L^{-1} F \in  \mathrm{Dom}L
= \mathbb{D}^{2,2}$,
and
\begin{equation}\label{Lmoins1}
LL^{-1} F = F - E(F).
\end{equation}

The following integration by parts formula is used throughout the paper.

\begin{lemma}\label{L : Tech1}
Suppose that $F\in\mathbb{D}^{1,2}$ and $G\in L^2(\Omega)$. Then, $L^{-1}G \in \mathbb{D}^{2,2}$ and
\begin{equation}
E[FG] = E[F]E[G]+E[\langle DF,-DL^{-1}G\rangle_{\HH}].
\end{equation}
\end{lemma}

\subsection{On cumulants}

The notion of cumulant will be crucial throughout the paper. We refer the reader to the monograph \cite{PecTaq} for an exhaustive discussion of such a notion.

\begin{definition}[Cumulants]\label{D : cum}{\rm Let $F$ be a real-valued random variable such that $E|F|^m<\infty$ for some integer
$m\geq 1$, and write $\phi_F(t) = E[e^{itF}]$, $t\in\R$, for the characteristic function of $F$.
Then, for $j=1,...,m$, the $j$th {\it cumulant} of $F$, denoted by $\kappa_j(F)$, is given by
\begin{equation}
\kappa_j (F) = (-i)^j \frac{d^j}{d t^j} \log \phi_F (t)|_{t=0}.
\end{equation}}
\end{definition}

\begin{rmk}{\rm When $E(F)=0$, then the first four cumulants of $F$ are the following: $\kappa_1(F) = E[F]=0$, $\kappa_2(F) = E[F^2]= {\rm Var}(F)$,
$\kappa_3(F) = E[F^3]$, and \[\kappa_4(F) = E[F^4] - 3E[F^2]^2.
\]
}
\end{rmk}
The following standard relation shows that moments can be recursively defined in terms of cumulants (and vice versa): fix $m= 1,2...$, and assume that $ E|F|^{m+1}<\infty$, then
\begin{equation}\label{EQ : RecMom}
 E[F^{m+1}] = \sum_{i=0}^m \binom{m}{i}\kappa_{i+1}(F)  E[F^{m-i}].
\end{equation}

Our aim is now to provide an explicit representation of cumulants in terms of Malliavin operators. To this end, it is convenient to introduce the following definition (see e.g. \cite[Chapter 8]{n-pe-1} for a full multidimensional version).

\begin{definition}\label{Def : Gamma} {\rm Let $F\in \mathbb{D}^{\infty}$. The sequence of random variables $\{\Gamma_i(F)\}_{i\geq 0}\subset
\mathbb{D}^\infty$ is recursively defined as follows. Set $\Gamma_0(F) = F$
and, for every $i\geq 1$, \[\Gamma_{i}(F) = \langle DF,-DL^{-1}\Gamma_{i-1}(F)\rangle_{\HH}.
\]}
\end{definition}

For instance, one has that $\Gamma_1(F) = \langle DF,-DL^{-1}F\rangle_{\HH}$. 
The following statement provides an explicit expression for $\Gamma_s(F)$, $s\geq 1$, when
$F$ has the form of a multiple integral.

\begin{propo}[See e.g. Chapter 8 in \cite{n-pe-1}]
\label{thm-pasmal}
Let $q\geq 2$, and assume that $F=I_q(f)$ with $f\in\HH^{\odot q}$.
Then, for any $i\geq 1$, we have
\begin{equation}
\begin{split}
& \Gamma_{i}(F)  = \\
& \sum_{r_1=1}^{q} \ldots\sum_{r_{i}=1}^{[iq-2r_1-\ldots-2r_{i-1}]\wedge q}c_q(r_1,\ldots,r_{i})
{\bf 1}_{\{r_1< q\}}
\ldots {\bf 1}_{\{
r_1+\ldots+r_{i-1}< \frac{iq}2
\}}\label{for2}\\
& \hskip3cm\times I_{(i+1)q-2r_1-\ldots-2r_{i}}\big(
(...(f\widetilde{\otimes}_{r_1} f) \widetilde{\otimes}_{r_2} f)\ldots
f)\widetilde{\otimes}_{r_{i}}f
\big),\notag
\end{split}
\end{equation}
where the constants $c_q(r_1,\ldots,r_{i-2})$ are recursively defined as follows:
\[
c_q(r)=q(r-1)!\binom{q-1}{r-1}^2,
\]
and, for $a\geq 2$,
\begin{eqnarray*}
&&c_q(r_1,\ldots,r_{a})\\
&& =\! q(r_{a}\!-\!1)! \! \binom{aq-2r_1-\ldots - 2r_{a-1}-1}{r_{a}-1}\!
\binom{q-1}{r_{a}-1}c_q(r_1,\ldots,r_{a-1}).
\end{eqnarray*}
\end{propo}

The following statement explicitly connects the expectation of the random variables $\Gamma_i(F)$ to the cumulants of $F$. 

\begin{propo}[See again Chapter 8 in \cite{n-pe-1}] Let $F\in \mathbb{D}^\infty$. Then $F$ has finite moments of every order, and the following relation holds for every $i\geq 0$:
\begin{equation}\label{e:cumgamma}
\kappa_{i+1}(F) = i! E[\Gamma_i(F)].
\end{equation}
\end{propo}

We also use the following result taken from \cite{b-b-n-p} throughout the paper.

\begin{lemma}\label{computation2}
Let $X \in \mathbb{D}^\infty$ Then, the relation
\begin{eqnarray}
 && E(\phi^{(k)}(X) \Gamma_{r}(X))\label{e:gibp}\\
 &&=  E(X \phi^{(k-r)}(X) ) - \sum_{s=1}^{r}  E(\phi^{(k-s)}(X))  E(\Gamma_{r-s}(X))\notag
\end{eqnarray} 
holds for every $k$-times continuously differentiable mapping $\phi: \R \to \R$. 
\end{lemma}

The next section is devoted to the elements of the second Wiener chaos.

\subsection{Some relevant properties of the second Wiener chaos}\label{ss:swc}

In this subsection, we gather together some properties the elements of the second Wiener chaos of the isonormal process $W=\{W(h); \ h \in \HH\}$; recall the these are random variables having the general form $F=I_2(f)$, with 
$f \in  \HH^{\odot 2}$. Notice that, if $f=h\otimes h$, where $h \in \HH$ is such that $\Vert h \Vert_{\HH}=1$, then using the multiplication formula (\ref{multiplication}), one has 
$I_2(f)=W(h)^2 -1  \stackrel{\text{law}}{=} N^2 -1$, where $N \sim \mathcal{N}(0,1)$. To any kernel $f \in \HH^{\odot 2}$, we associate the following \textit{Hilbert-Schmidt} operator
\begin{equation*}
A_f : \HH \mapsto \HH; \quad g \mapsto f\otimes_1 g. 
\end{equation*}
It is also convenient to introduce the sequence of auxiliary kernels
\begin{equation}
\left\{ f\otimes _{1}^{\left( p\right) }f:p\geq 1\right\} \subset \mathfrak{H%
}^{\odot 2}  \label{kern1}
\end{equation}%
defined as follows: $f\otimes _{1}^{\left( 1\right) }f=f$, and, for $p\geq 2$,%
\begin{equation}
f\otimes _{1}^{\left( p\right) }f=\left( f\otimes _{1}^{\left(
p-1\right) }f\right) \otimes _{1}f\text{.}  \label{kern2}
\end{equation}
In particular,
$f\otimes _{1}^{\left( 2\right) }f=f\otimes _1 f$. Finally, we write $\{\alpha_{f,j}\}_{j \ge 1}$ and $\{e_{f,j}\}_{j \ge 1}$, respectively, to indicate the (not necessarily distinct) eigenvalues of $A_f$ and the corresponding eigenvectors.

\begin{propo}[See e.g. Section 2.7.4 in \cite{n-pe-1}] \label{second-property}
Fix $F=I_2(f)$ with $f \in \HH^{\odot 2}$.
\begin{enumerate}
 \item The following equality holds: $F=\sum_{j\ge 1} \alpha_{f,j} \big( N^2_j -1 \big)$, where $\{N_j\}_{j \ge 1}$ is a sequence of i.i.d. $\mathcal{N}(0,1)$ random variables that are elements of the isonormal process $W$, and the series converges in $L^2$ and almost surely.
 
 \item For any $i \ge 2$,
 \begin{equation*}
  \kappa_i(F)= 2^{i-1}(i-1)! \sum_{j \ge 1} \alpha_{f,j}^i = 2^{i-1}(i-1)! \times \langle f \otimes^{(i-1)}_{1}f ,f \rangle_{\HH^{\otimes 2}}.
 \end{equation*}
\item The law of the random variable $F$ is completely determined by its moments or equivalently by its cumulants.
\end{enumerate}

\end{propo}

\section{Main results}

Throughout this section, we assume that $\{ W(h):\ h \in \HH\}$ is a centered isonormal Gaussian process on a separable Hilbert space $\HH$ having $\{e_i\}_{i \ge 1}$ as a complete orthonormal basis.

\subsection{A new view of reference \cite{n-po-1}}\label{ss:np}
We now fix a symmetric kernel $f_\infty \in \HH^{\odot 2}$ such that the its corresponding Hilbert-Schmidt operator $A_{f_\infty}$ (see Section \ref{ss:swc}) has a finite number of non-zero eigenvalues, that we denote by $\{ \alpha_i\}_{i=1}^{k}$. To simplify the discussion, we assume that the eigenvalues are all distinct. As anticipated, we want to study convergence in distribution towards the random variable

\begin{equation}\label{target-wiener}
  F_{\infty}:= I_2(f_\infty)= \sum_{i=1}^k \alpha_i \left(N_i^2-1\right),
\end{equation}
where $\{N_i\}_{i=1}^{k}$ is the family of i.i.d. $\mathcal{N}(0,1)$ random variables appearing at Point 1 of Proposition \ref{second-property}. Following Nourdin and Poly \cite{n-po-1}, we define the two crucial polynomials $P$ and $Q$ as follows:
\begin{equation}\label{polynomialP}
 Q(x)=\big( P(x)\big)^{2}=\Big(x \prod_{i=1}^{k}(x - \alpha_i) \Big)^{2}.
\end{equation}
Note that, by definition, the roots of $Q$ and $P$ correspond with the set $\{0, \alpha_1,...,\alpha_k\}$.

The starting point of our discussion is the following result, proved in \cite{n-po-1}: it provides necessary and sufficient conditions for a sequence in the second Wiener chaos of $W$ to converge in distribution towards $F_\infty$.

\begin{thm}[See \cite{n-po-1}]\label{t:noup}
 Consider a sequence $\{ F_n \}_{n \ge 1} = \{ I_2 (f_n)\}_{n \ge 1}$ of double Wiener integrals with $f_n \in \HH^{\odot 2}$. Then, the following statements are equivalent, as $n\to \infty$:
\begin{description}
 \item[(i)] 
 $F_n \stackrel{{\rm law}}{\to} F_\infty$;

\item[(ii)]
the following two asymptotic relations are verified:
\begin{enumerate}
 \item $\kappa_r (F_n) \to \kappa_r (F_\infty), \quad \text{for all} \ 2\le r \le k+1=\text{deg}(P)$,
 \item $ \sum_{r=2}^{\text{deg}(Q)} \frac{Q^{(r)}(0)}{r!} \frac{\kappa_{r}(F_n)}{2^{r-1}(r-1)!} \to 0$.
 \end{enumerate}
\end{description}
\end{thm}

\medskip

The original proof of Theorem \ref{t:noup} is based on methods from complex analysis, and exploits the fact that (owing to the representation stated at Point 1 of Proposition \ref{second-property}) the 
Fourier transform of a random variable with the form $I_2(f)$ can be written down explicitly in terms of the eigenvalues $\{\alpha_{f,j}\}$. Our aim in this section is to prove that condition {\bf (ii)} of 
Theorem \ref{t:noup} can be equivalently stated in terms of contractions and Malliavin operators. Such equivalent conditions naturally lead to the main findings of the paper, as stated in Theorem 
\ref{main-general-thm}, that will also provide (as a by-product) an alternate proof of Theorem \ref{t:noup} that does not make use of complex analysis (see, in particular, Remark \ref{r:r} below). We start 
with a crucial lemma, that is in some sense the linchpin of the whole paper.

\begin{lemma}\label{lemma1}
 Let $F=I_{2}(f)$, $f \in \HH^{ \odot 2}$, be a generic element of the second Wiener chaos of $W$, and write $\{\alpha_{f,j}\}_{j\geq 1}$ for the set of the eigenvalues of the associated Hilbert-Schmidt operator $A_f$ we have
 \begin{eqnarray}
 && \sum_{r=2}^{\text{deg}(Q)} \frac{Q^{(r)}(0)}{r!} \frac{\kappa_{r}(F)}{2^{r-1}(r-1)!} \notag\\
 &&= \sum_{j\geq 1} Q(\alpha_{f,j})\label{e:2}\\
 && = \Bigg\Vert \sum_{r=1}^{\text{deg}(P)} \frac{P^{(r)}(0)}{r!} f \otimes^{(r)}_{1} f \Bigg\Vert^{2}_{\HH^{ \otimes 2}} \label{e:3}\\
 && = \frac12E \Bigg( \sum_{r=1}^{\text{deg}(P)} \frac{P^{(r)}(0)}{r! \ 2^{r-1}} \Big( \Gamma_{r-1}(F) -   E(\Gamma_{r-1}(F)) \Big) \Bigg)^2,\label{e:4}
 \end{eqnarray}
 where the operators $\Gamma_{r}(\cdot)$ have been introduced in Definition \ref{Def : Gamma}. In particular, for the target random variable $F_\infty$ introduced at \eqref{target-wiener} one has that
\begin{eqnarray}
 && 0 =\sum_{r=2}^{\text{deg}(Q)} \frac{Q^{(r)}(0)}{r!} \frac{\kappa_{r}(F_\infty)}{2^{r-1}(r-1)!} \notag\\
&& = \frac12E \Bigg( \sum_{r=1}^{\text{deg}(P)} \frac{P^{(r)}(0)}{r! \ 2^{r-1}} \Big( \Gamma_{r-1}(F_\infty) -   E(\Gamma_{r-1}(F_\infty)) \Big) \Bigg)^2.\label{e:t0}
 \end{eqnarray}
\end{lemma}
\begin{proof}
In view of the second equality at Point 2 of Proposition \ref{second-property}, one has that $\frac{\kappa_{r}(F)}{2^{r-1}(r-1)!}  = \sum_{j\geq 1} \alpha_{f,j}^r$, from which we deduce immediately \eqref{e:2}. To prove \eqref{e:3}, observe that Point 1 of Proposition \ref{second-property}, together with the product formula \eqref{multiplication}, implies that the kernel $f$ admits a representation of the type $f = \sum_{j\geq 1} \alpha_{f,j} \eta_j \otimes \eta_j$, where $\{\eta_j\}$ is some orthonormal system in $\HH$. It follows that, for $r\geq 1$, one has the representation $f\otimes_1^{(r)} f = \alpha^r_{f,j} \eta_j \otimes \eta_j$, and therefore 
\begin{eqnarray*}
&& \sum_{r=1}^{\text{deg}(P)} \frac{P^{(r)}(0)}{r!} f \otimes^{(r)}_{1} f = \sum_{j\geq 1} \eta_j \otimes \eta_j  \sum_{r=1}^{\text{deg}(P)} \frac{P^{(r)}(0)}{r!}\alpha^r_{f,j}.
\end{eqnarray*}
Taking norms on both sides of the previous relation and exploiting the orthonormality of the $\eta_j$ yields \eqref{e:3}. Finally, in order to show \eqref{e:4}, it is clearly enough to prove that, for any $r \ge 1$,
\begin{equation}\label{Gamma-indentity}
I_2(f \otimes^{(r)}_{1} f) = \frac{1}{2^{r-1}}\big\{ \Gamma_{r-1}(F) -  E(\Gamma_{r-1}(F)) \big\}.
\end{equation}
We proceed by induction on $r$. It is clear for $r=1$, because $\Gamma_0(F)=F$ and $ E(F)=0$. Take $r \ge 2$ and assume that $(\ref{Gamma-indentity})$ holds true. Without loss of generality, we can assume that 
$\HH=L^2 (A,\mathcal{A},\mu)$, where $\mu$ is a $\sigma$-finite and non-atomic measure on the measurable space $(A,\mathcal{A})$. Notice that, by definition of $\Gamma_r(F)$ and the induction assumption, one has
\begin{equation*}
 \begin{split}
&  \Gamma_r(F)\\
&= \langle DF,-D\LL^{-1}\Gamma_{r-1}(F)\rangle_{\HH}=\left \langle 2 I_1(f(t,.)),2^{r-1} I_1(f \otimes^{(r)}_{1} f(t,.))\right\rangle_{\HH}\\
 &= 2^r \int_{A} \Big\{ \langle f(t,.),f \otimes^{(r)}_{1} f(t,.)\rangle_{\HH} + I_2\big(f(t,.)\otimes (f \otimes^{(r)}_{1} f)(t,.)\big) \Big\} \ud\mu(t)\\
 &= 2^r \langle f, f \otimes^{(r)}_{1} f\rangle_{\HH^{\odot 2}} + 2^r I_2(f \otimes^{(r+1)}_{1} f),
 \end{split}
\end{equation*}
where we have used a standard stochastic Fubini Theorem. This proves that \eqref{Gamma-indentity} is verified for every $r\geq 1$. The last assertion in the statement follows from \eqref{e:2}, as well as the fact that the eigenvalues $\alpha_i$ are all roots of $Q$.
\end{proof}

The next proposition, which is an immediate consequence of Lemma \ref{lemma1}, provides the announced extension of Theorem \ref{t:noup}.

\begin{propo}\label{p:z}
Assume $\{ F_n \}_{n \ge 1} = \{ I_2 (f_n)\}_{n \ge 1}$ be a sequence of double Wiener integrals with $f_n \in \HH^{\odot 2}$. Then the following statements are equivalent to either Point {\bf (i)} or {\bf (ii)} of Theorem \ref{t:noup}, as $n\to\infty$.
\begin{description}
 \item[(a)] The following relations 1.-2. are in order: 
 \begin{enumerate}
 \item $\kappa_r (F_n) \to \kappa_r (F_\infty), \quad \text{for all} \ 2\le r \le k+1=\text{deg}(P)$, and
\item $  E \Bigg( \sum_{r=1}^{k+1} \frac{P^{(r)}(0)}{r! \ 2^{r-1}} \Big( \Gamma_{r-1}(F) -   E(\Gamma_{r-1}(F)) \Big) \Bigg)^2 \to 0$.
\end{enumerate}
\item[(b)] The following relations 1.-2. are in order:
\begin{enumerate}
 \item $\kappa_r (F_n) \to \kappa_r (F_\infty), \quad \text{for all} \ 2\le r \le k+1=\text{deg}(P)$, and
 \item $\Bigg\Vert \sum_{r=1}^{\text{deg}(P)} \frac{P^{(r)}(0)}{r!} f_{n} \otimes^{(r)}_{1} f_{n} \Bigg\Vert^{2}_{\HH^{\otimes 2}} \to 0$.
 \end{enumerate}
\end{description}
\end{propo}

\medskip

As anticipated, our aim in the sections to follow is to show that the equivalence between Condition {\bf (a)} in Proposition \ref{p:z} and Condition {\bf (i)} in Theorem \ref{t:noup} is indeed valid for sequence of random variables living in a finite sum of Wiener chaoses. The next statement provides a first, non dynamical version of this fact.

\begin{propo}\label{p:static}
Let the polynomial $P$ be defined as in \eqref{polynomialP} and consider again the random variable $F_\infty = I_2(f_\infty)$ defined in \eqref{target-wiener}. Let $F$ be a centered random variable living in a finite sum of Wiener chaoses, i.e. 
$F \in \bigoplus_{i=1}^{M} C_i$. Moreover, assume that
\begin{itemize}
 \item[\bf (i)]$\kappa_r (F) = \kappa_r (F_\infty)$, for all $2 \le r \le k+1=\text{deg}(P)$, and
\item[\bf (ii)] 

\begin{equation*}
              E \Bigg( \sum_{r=1}^{k+1} \frac{P^{(r)}(0)}{r! \ 2^{r-1}} \Big( \Gamma_{r-1}(F) -   E(\Gamma_{r-1}(F)) \Big) \Bigg)^2 = 0.
            \end{equation*}
\end{itemize}
Then, $F \stackrel{{\rm law}}{=} F_\infty,$ and $F\in C_2$.
\end{propo}

\begin{proof}
Let $\phi$ be a smooth function. Using the integration by parts formula (Lemma \ref{L : Tech1}) and Assumption $\bf{(ii)}$ in the statement, we obtain 
\begin{equation}\label{computation1}
\begin{split}
 E \big( F \phi(F) \big) &= \sum_{r=0}^{k-1} \frac{\kappa_{r+1}(F)}{r!}  E(\phi^{(r)}(F)) +  E(\phi^{(k)}(F) \Gamma_{k}(F))\\
&=\sum_{r=0}^{k-1} \frac{\kappa_{r+1}(F)}{r!}  E(\phi^{(r)}(F)) + \frac{\kappa_{k+1}(F)}{k!} E(\phi^{(k)}(F))\\
&+\sum_{r=1}^{k} \frac{2^{k-r+1}\kappa_{r}(F)}{(r-1)!r!} P^{(r)}(0)  E(\phi^{(k)}(F))\\
&- \sum_{r=1}^{k} \frac{2^{k-r+1}}{r!} P^{(r)}(0)  E(\phi^{(k)}(F) \Gamma_{r-1}(F))
\end{split}
\end{equation}
On the other hand, using \eqref{e:gibp} we obtain that
\begin{equation}\label{computation2}
\begin{split}
  E(\phi^{(k)}(F) \Gamma_{r-1}(F))&=  E(F \phi^{(k-(r-1))}(F) )\\
 &\quad- \sum_{s=1}^{r-1}  E(\phi^{(k-s)}(F))  E(\Gamma_{r-1-s}(F)).
 \end{split}
\end{equation}
Using the relation $ E(\Gamma_{r-1-s}(F)) = \kappa_{r-s}(F)/(r-s-1)!$, therefore deduce that, for every smooth test function $\phi$
\begin{eqnarray*}
E \big( F \phi(F) \big) &=&\sum_{r=0}^{k-1} \frac{\kappa_{r+1}(F)}{r!}  E(\phi^{(r)}(F)) + \frac{\kappa_{k+1}(F)}{k!} E(\phi^{(k)}(F))\\
&&+\sum_{r=1}^{k} \frac{2^{k-r+1}\kappa_{r}(F)}{(r-1)!r!} P^{(r)}(0)  E(\phi^{(k)}(F))\\
&&- \sum_{r=1}^{k} \frac{2^{k-r+1}}{r!} P^{(r)}(0)E[F\varphi^{(k-(r-1))}(F)]\\
&&+\sum_{r=1}^{k} \frac{2^{k-r+1}}{r!} P^{(r)}(0)\sum_{s=1}^{r-1}E[\phi^{(k-s)}(F)] \frac{\kappa_{r-s}(F)}{(r-s-1)!}. 
\end{eqnarray*}
Considering the test function $\phi(x)=x^n$ with $n > k$, we infer that $ E(F^{n+1})$ can be expressed in a recursive way in terms of the quantities $$ E(F^n),  E(F^{n-1}), \cdots, E(F^{n-k}), \,
\kappa_2(F), \cdots, \kappa_{k+1}(F)$$ and $P^{(1)}(0),$ $\cdots, P^{(k)}(0)$. Using Assumption ${\bf (i)}$ in the statement together with last assertion in Lemma \ref{lemma1}, we see that the moments of the random variable $F_\infty$ also satisfy the same recursive relation. These facts immediately imply that
\begin{equation*}
 E \left( F^n \right) =  E \left( F_\infty^n \right), \quad n \ge 1,
\end{equation*}
and the claim follows at once from Point 3 in Proposition \ref{second-property}. To prove that, in fact, $F \in C_2$, we assume that $M$ is the smallest natural number such that $F \in \bigoplus_{i=1}^{M} C_i$. Hence 
 $F \notin \bigoplus_{i=1}^{M-1} C_i$. Therefore, by applying \cite[Theorem 6.12]{jonson} to $F$, $F_\infty$ and the fact that $F \stackrel{{\rm law}}{=} F_\infty$, we deduce that $M=2$. Let assume that 
 $F = I_1(g)+ I_2(h)$ for some  $g \in \HH$ and $h \in \HH^{\otimes 2}$. Considering the trivial sequence $\{F_n\}_{n \ge 1}$ such that $F_n = F_\infty$, $n\geq 1$, using the fact that 
 $F \stackrel{{\rm law}}{=} F_\infty$ and applying \cite[Theorem 3.1]{n-po-1}, we deduce that $I_1(g)$ is independent of $I_2(h)$. Let $\{\lambda_{f_\infty,k}\}_{k \ge 1}$ and $\{ \lambda_{h,k}\}_{k \ge 1}$ denote the eigenvalues 
  corresponding to the Hilbert-Schmidt operator $A_{f_\infty}$ and $A_{h}$ associated with the kernels $f_\infty$ and $h$ respectively (see Section \ref{ss:swc}). Exploiting the independence of $I_1(g)$ and 
  $I_2(h)$ and Point 3 in Proposition \ref{second-property}, we infer that $$\sum_{k \in \N} \lambda_{f_\infty,k}^{3p} = \sum_{k \in \N} \lambda_{h,k}^{3p} \quad \forall \, p \ge 1.$$ As result, Lemma 
  \ref{lemma-appendix} in Appendix implies that for some permutation $\pi$ on $N$ we have $\lambda_{\infty,k} = \lambda_{h,\pi(k)}$ for $k \ge 1$, which in turn implies  
  \begin{equation}\label{key}
  \sum_{k \in \N} \lambda_{f_\infty,k}^{2} = \sum_{k \in \N} \lambda_{h,k}^{2}.
  \end{equation}
  On the other hand, from $F = I_1(g)+ I_2(h)  \stackrel{{\rm law}}{=} F_\infty$, and computing the second cumulant of both sides, one can easily deduce that if 
  $\kappa_{2}(I_{1}(g))= \E(I_{1}(g))^2= \Vert g \Vert_{\HH}^2 \neq 0$, then the equality  $(\ref{key})$ cannot hold. It follows that $I_1(g)=0$, and therefore $F\in C_2$. 
\end{proof}
\medskip

One of the arguments used in the previous proof will be exploited again in the next section. For future reference, we shall explicitly state the needed double implication in the form of a Lemma.

\begin{lemma}\label{main-lemma-2} Let $F$ be a centered random variable, with finite moments of all orders and such that $\kappa_r (F) = \kappa_r (F_\infty)$, for all $2 \le r \le k+1=\text{deg}(P)$. Then, $F \stackrel{{\rm law}}{=} F_\infty$ if and only if, for every polynomial mapping $\varphi : \R\to \R$,
\begin{eqnarray}
&&E \big( F \phi(F) \big)\label{e:lemmino} \\
&&=\Psi_\phi(F):= \sum_{r=0}^{k-1} \frac{\kappa_{r+1}(F)}{r!}  E(\phi^{(r)}(F)) + \frac{\kappa_{k+1}(F)}{k!} E(\phi^{(k)}(F))\notag \\
&&+\sum_{r=1}^{k} \frac{2^{k-r+1}\kappa_{r}(F)}{(r-1)!r!} P^{(r)}(0)  E(\phi^{(k)}(F))\notag \\
&&- \sum_{r=1}^{k} \frac{2^{k-r+1}}{r!} P^{(r)}(0)E[F\varphi^{(k-(r-1))}(F)]\notag \\
&&+\sum_{r=1}^{k} \frac{2^{k-r+1}}{r!} P^{(r)}(0)\sum_{s=1}^{r-1} E[\phi^{(k-s)}(F)]\frac{\kappa_{r-s}(F)}{(r-s-1)!}. \notag
\end{eqnarray}
\end{lemma}

\medskip

In the next section, which contains the main findings of the paper, we shall show that a slight variation of Condition {\bf (a)} in Proposition \ref{p:z} is basically necessary and sufficient for convergence in distribution towards $F_\infty$ for {\it any} sequence of random variables living in a finite sum of Wiener chaoses.

\subsection{A general criterion}

We recall that the \textit{total variation} distance $d_{\text{TV}}$ between the laws of two real-valued random variables $X$ and $Y$ is defined as

\begin{equation}
d_{\text{TV}} (F,G) = \sup_{A \in \mathcal{B}(\R)} \Big \vert \P (F \in A) - \P(G \in A) \Big \vert,
\end{equation}
where the supremum is taken over all Borel sets $A \subseteq \R$. We also write $ F_n \xrightarrow{{\rm TV}}F$ to indicate the asymptotic relation $d_{\rm TV}(F_n, F)\to 0$. \\

The next theorem is the main finding of the paper. Recall that the random variable $F_\infty$ has been defined in formula (\ref{target-wiener}). 

\begin{thm}\label{main-general-thm}
Let $\{F_n\}_{n\ge 1}$ be a sequence of random variables such that each $F_n$ lives in a finite sum of chaoses, i.e. $F_n \in \bigoplus_{i=1}^{M}C_i$ for $n \ge 1$ and some $M\geq 2$ (not depending on $n$). Consider the following three asymptotic relations, as $n\to \infty$:

\begin{description}

 \item[(i)] 
 \begin{equation}\label{Contv}
F_n\xrightarrow{{\rm TV}}~ F_\infty;
\end{equation}

\item[(ii)] The following relations 1.-2. are in order:
\begin{enumerate}\label{expcondi}
 \item $\kappa_r (F_n) \to \kappa_r (F_\infty), \quad \text{for all} \ 2\le r \le k+1=\text{deg}(P)$, and
 \item $ E \Bigg( \sum_{r=1}^{k+1}\frac{P^{(r)}(0)}{r! 2^{r-1}}\Big(\Gamma_{r-1}(F_n)- E[\Gamma_{r-1}(F_n)]\Big) \Bigg\vert F_n \Bigg) \xrightarrow{L^2}~0$.
\end{enumerate}

\item[(iii)] The following relations 1.-2. are in order:
\begin{enumerate}\label{expcondi2}
 \item $\kappa_r (F_n) \to \kappa_r (F_\infty), \quad \text{for all} \ 2\le r \le k+1=\text{deg}(P)$, and
 \item $ E \Bigg( \sum_{r=1}^{k+1}\frac{P^{(r)}(0)}{r! 2^{r-1}}\Big(\Gamma_{r-1}(F_n)- E[\Gamma_{r-1}(F_n)]\Big) \Bigg\vert F_n \Bigg) \xrightarrow{L^1}~0$.
\end{enumerate}

\end{description}

Then, one has the implications {\bf (ii)} $\to$ {\bf (i)} and {\bf (i)} $\to$ {\bf (iii)}.


\end{thm}

 \begin{rmk}{\rm We remark that, in the special case  $k=1=\alpha_k $, the condition appearing at Point 2 of item {\bf(ii)} in Theorem \ref{main-general-thm} is implied by the relation 
 
 \begin{equation}\label{e:uf}
 E \big( \Gamma_1(F_n) - F_n - 2 \big)^2 \to 0.
 \end{equation}
 When $F_n =I_q(f_n)$, this corresponds to the condition appearing at Point (iii) of Part (B) of Theorem \ref{T:NPEC}, by taking into account the fact that, for a multiple integral $F=I_q(f)$ of order $q$, we have the relation 
  $\Gamma_1(F)=\frac{1}{q} \Vert DF \Vert_{\HH}^2$. Note that, as explained in \cite{n-pe-2}, the asymptotic relation \eqref{e:uf} cannot be fulfilled by a sequence $F_n$ such that $F_n =I_q(f_n)$ with $q$ odd and $E[F_n^2]\to 2$.}
 \end{rmk}

In order to prove Theorem \ref{main-general-thm}, we need an additional lemma.

\begin{lemma}[See Theorem 3.1 in \cite{n-po-2}]\label{hyper}
Let $\{ F_n\}_{n \ge 1}$ be a sequence of non-zero random variables living in a finite sum of Wiener chaoses, i.e.  $F_n \in \bigoplus_{i=0}^{M} C_i {,\,\,\forall n\ge 1}$. Assume that the 
sequence $\{F_n\}_{n \ge 1}$ converges in distribution to some non-zero target random variable $F$, as $n$ tends to infinity. Then, 
\begin{equation}\label{e:hyper}
\sup_{n \ge 1}  E(\vert F_n \vert^r) < \infty, \quad \forall \ r \ge 1,
\end{equation}
and $F_n\xrightarrow{{\rm TV}}~ F$. Moreover, the distribution of $F$ is necessarily absolutely continuous with respect to the Lebesgue measure.
\end{lemma}

\medskip

\noindent \textit{Proof of Theorem \ref{main-general-thm}}. 

\smallskip

\noindent [Proof of $\bf(ii) \to \bf(i)$] Assumption 1 in $\bf(ii)$ implies that $\sup_{n \ge 1} E\left(F_n^2\right) < \infty$. Hence, the sequence $\{F_n\}_{n \ge 1}$ is tight. This yields that, for any subsequence 
$\{F_{n_{k}}\}_{k \ge 1}$, there exists a sub-subsequence $\{F_{n_{k_{l}}}\}_{l \ge 1}$ and a random variable $F$ such that $F_{n_{k_{l}}} \stackrel{{\rm law}}{\rightarrow} F $, as $l$ tends to infinity. In order to show the desired implication, we have now to show that, necessarily, $F$ has the same distribution as $F_\infty$. To simplify the discussion, we may assume that $\{F_{n_{k_{l}}}\}_{l \ge 1}=\{F_n\}_{n \ge 1}$. By exploiting \eqref{e:hyper} together with the fact that the sequence $\{F_n\}_{n\geq 1}$ lives in a fixed finite sum of Wiener chaoses, we deduce that, for every polynomial $\phi$,
\begin{equation}
E\left( F_n \phi(F_n) \right) \to  E \left( F \phi(F) \right), \quad n \to \infty.
\end{equation}
and
\begin{eqnarray}
&&\Psi_{\phi}(F_n) \longrightarrow \Psi_{\phi}(F), \quad \text{as} \ n \to \infty, \label{e:b}
\end{eqnarray} 
where we have used the notation  \eqref{e:lemmino}. By virtue of Lemma \ref{main-lemma-2}, in order to show the desired implication, it is then sufficient to prove the asymptotic relation
\begin{equation}\label{claim}
\Big \vert  E\left( F_n \phi(F_n) \right) - \Psi_{\phi}(F_n) \Big \vert \to 0, \quad n\to\infty,
\end{equation}
for every polynomial $\phi$. To show \eqref{claim}, we can use several times integration by parts (see Lemma \ref{L : Tech1}) to infer that
\begin{eqnarray*}
&&\Big \vert  E\left( F_n \phi(F_n) \right)  - \Psi_{\phi}(F_n) \Big \vert \\
&& = 2^k  E \Bigg[ \phi^{(k)}(F_n)\Bigg( \sum_{r=1}^{k+1} \frac{P^{(r)}(0)}{r! \ 2^{r-1}} \Big( \Gamma_{r-1}(F_n) -   E(\Gamma_{r-1}(F_n)) \Big) \Big \vert F_n\Bigg) \Bigg]\\
&& \le 2 ^k \sqrt{ E\left( \phi^{(k)}(F_n)\right)^2} \times \\
&&\quad\quad\quad \times \sqrt{ E \Bigg( \sum_{r=1}^{k+1} \frac{P^{(r)}(0)}{r! \ 2^{r-1}} \Big( \Gamma_{r-1}(F_n) -   E(\Gamma_{r-1}(F_n)) \Big) \Big \vert F_n \Bigg)^2}.
\end{eqnarray*}
Now, a standard application of Lemma \ref{hyper} shows that $$\sup_{n \ge 1} E \left( \phi^{(k)}(F_n)\right)^2 < \infty,$$ and \eqref{claim} follows by exploiting Assumption $2$ at Point {\bf (ii)}.\\

\smallskip

\noindent[Proof of $\bf(i) \to (iii)$] The proof is divided into several steps. Take $\phi\in\mathcal{C}^\infty_c$ with support in $[-M,M]$ where $M>0$ and $\|\phi^{(k)}\|_\infty\leq1$.~\\ 
\underline{Step 1.} We have:
\begin{eqnarray*}
&& E\Bigg(\phi^{(k)}(F_n )  \sum_{r=1}^{k+1}\frac{P^{(r)}(0)}{r! 2^{r-1}}\Big(\Gamma_{r-1}(F_n )- E(\Gamma_{r-1}(F_n ))\Big)\Bigg)\\
&&\quad\quad\quad\quad= \sum_{r=1}^{k+1}\frac{P^{(r)}(0)}{r! 2^{r-1}}  E(\phi^{(k)}(F_n ) \Gamma_{r-1}(F_n ))
 \\
&& \quad\quad\quad\quad \quad\quad-  E[\phi^{(k)}(F_n)] \sum_{r=1}^{k+1}\frac{P^{(r)}(0)}{r! 2^{r-1}}  E(\Gamma_{r-1}(F_n))\\
&&\quad\quad\quad\quad = \sum_{r=1}^{k+1}\frac{P^{(r)}(0)}{r! 2^{r-1}}  E(F_n  \phi^{(k-(r-1))}(F_n ) ) \\
&&\quad\quad\quad\quad\quad\quad - \sum_{r=1}^{k+1}\frac{P^{(r)}(0)}{r! 2^{r-1}} \sum_{s=1}^{r-1}  E(\phi^{(k-s)}(F_n ))  E(\Gamma_{r-1-s}(F_n ))\\
&&\quad\quad\quad\quad\quad\quad\quad\quad  -  E(\phi^{(k)}(F_n )) \sum_{r=1}^{k+1}\frac{P^{(r)}(0)}{r! 2^{r-1}}  E(\Gamma_{r-1}(F_n ))\\
&&\quad\quad\quad\quad=\sum_{r=0}^{k} E\left(\phi^{(r)}(F_n )\left(A_{r,n} F_n +B_{r,n}\right)\right).
\end{eqnarray*}
Here $\{A_{r,n},B_{r,n}\}_{0\le r\le k}$ are constants continuously depending only on the $k+1$ first cumulants of $F_n $. Since (\ref{Contv}) holds and since $\{F_n \}_{n \ge1}$ is bounded in $L^{k+1}(\Omega)$, for each $r\in\{1,2,\cdots,k+1\}$ we have 
$$\kappa_r(F_n )\to2^{r-1} (r-1)!\sum_{i=1}^k \alpha_i^r =\kappa_r\left(F_\infty\right).$$
This yields that
\begin{equation}\label{fondamentalfact}
\begin{array}{l}
 E\Bigg(\phi^{(k)}(F_n )  \sum_{r=1}^{k+1}\frac{P^{(r)}(0)}{r! 2^{r-1}}\Big(\Gamma_{r-1}(F_n )- E(\Gamma_{r-1}(F_n ))\Big)\Bigg)\\
\to E\Bigg(\phi^{(k)}(F_\infty)  \sum_{r=1}^{k+1}\frac{P^{(r)}(0)}{r! 2^{r-1}}\Big(\Gamma_{r-1}(F_\infty)- E(\Gamma_{r-1}(F_\infty))\Big)\Bigg) = 0,
\end{array}
\end{equation}
where we have used Lemma \ref{lemma1}.
\vskip0.5cm

%
%
\noindent\underline{Step 2.} The conclusion at Step 1 implies that, for each fixed $\phi\in\mathcal{C}^\infty_c$ with support in $[-M,M]$, such that $\|\phi^{(k)}\|_\infty\leq1$, we have:
\begin{equation}\label{fondamentalfact2}
 E\Bigg(\phi^{(k)}(F_n )  \sum_{r=1}^{k+1}\frac{P^{(r)}(0)}{r! 2^{r-1}}\Big(\Gamma_{r-1}(F_n )- E(\Gamma_{r-1}(F_n ))\Big)\Bigg)\to 0.
\end{equation}
For convenience we set 
$$\mathcal{E}_M=\Big\{\phi\in\mathcal{C}_c^\infty\Big|\,\|\phi^{(k)}\|_\infty\leq 1,\,\text{supp}(\phi)\subset[-M,M]\Big\}.$$
Exploiting again the arguments used in Step 1 we infer that
\begin{eqnarray*}
 &&E\Bigg(\phi^{(k)}(F_n )  \sum_{r=1}^{k+1}\frac{P^{(r)}(0)}{r! 2^{r-1}}\Big(\Gamma_{r-1}(F_n )- E(\Gamma_{r-1}(F_n ))\Big)\Bigg)\\
 &&=\sum_{r=0}^{k} E\left(\phi^{(r)}(F_n )\left(A_{r,n} F_n +B_{r,n}\right)\right).\\
\end{eqnarray*}
One has that
\begin{eqnarray*}
&&\sup_{\phi\in\mathcal{E}_M}\left| \sum_{r=0}^{k} E\left(\phi^{(r)}(F_n )\left(A_{r,n} F_n +B_{r,n}\right)\right)\right.\\
&&\quad\quad\quad\quad\quad\left. -\sum_{r=0}^{k} E\left(\phi^{(r)}(F_n )\left(A_{r,\infty} F_n +B_{r,\infty}\right)\right)\right|\\
&\leq&\sup_{\phi\in\mathcal{E}_M}\sum_{r=0}^k\|\phi^{(r)}\|_\infty\left(|A_{r,n}-A_{r,\infty}|\sup_{n\ge 1} E(|F_n |)+|B_{r,n}-B_{r,\infty}|\right)\\
&\leq& M^k\left( \sup_{n\ge 1} E( |F_n |)+1\right) \sum_{r=0}^k\left(|A_{r,n}-A_{r,\infty}|+|B_{r,n}-B_{r,\infty}|\right)\\
&\to 0,&
\end{eqnarray*}
where we have used the fact that for $\phi\in\mathcal{E}_M$, and for any $0\le r \le k$, we have $\|\phi^{(r)}\|_\infty \le M^k$. On the other hand, we know that $F_n  \xrightarrow[n\to\infty]{TV} F_\infty$. The following equality holds
\begin{eqnarray*}
&& \left|\sum_{r=0}^{k} E\left(\phi^{(r)}(F_n )\left(A_{r,\infty} F_n +B_{r,\infty}\right)\right)\right|\\
&&=\left|\sum_{r=0}^{k} E\left(\phi^{(r)}(F_n )\left(A_{r,\infty} F_n +B_{r,\infty}\right)\right)\right.\\
&& \left.-\sum_{r=0}^{k} E\left(\phi^{(r)}(F_\infty)\left(A_{r,\infty} F_\infty+B_{r,\infty}\right)\right)\right|.
\end{eqnarray*}
The expression on the right-hand side of the previous equality is bounded by
\begin{eqnarray*}
&&\sum_{r=0}^{k} A_{r,\infty}\left| E\left(\phi^{(r)}(F_n )F_n -\phi^{(r)}(F_\infty)F_\infty\right)\right|\\
&& +\sum_{r=0}^{k} B_{r,\infty} \left| E\left(\phi^{(r)}(F_n )-\phi^{(r)}(F_\infty)\right)\right|\\
&\le&{M^{k+1} \left(\sum_{r=0}^k A_{r,\infty}+B_{r,\infty}\right)d_{TV}(F_n,F_\infty).}
\end{eqnarray*}
To obtain the previous estimate, we have used the facts that $$\sup_{x\in[-M,M]}\|\phi^{(r)}(x) x\|\le M^{k+1} \quad \text{and} \quad \|\phi^{(r)}\|_\infty \le M^k.$$ 
Now, letting $n\to\infty$, we deduce that
\begin{equation}\label{Almostdone}
\sup_{\phi\in\mathcal{E}_M}\left|\sum_{r=0}^{k} E\left(\phi^{(r)}(F_n )\left(A_{r,\infty} F_n +B_{r,\infty}\right)\right)\right|\to 0,
\end{equation}
as well as
\begin{equation}\label{Almostdone2}
\sup_{\phi\in\mathcal{E}_M}\left|\sum_{r=0}^{k} E\left(\phi^{(r)}(F_n )\left(A_{r,n} F_n +B_{r,n}\right)\right)\right|\to 0.
\end{equation}
~\\
\underline{Step 3.} Let $\mathcal{F}_M$ be the set of Borel functions bounded by $1$ and supported in $[-M,M]$. By density we have
\begin{eqnarray*}
&&\sup_{\phi\in\mathcal{E}_M}\left|\sum_{r=0}^{k} E\left(\phi^{(r)}(F_n )\left(A_{r,n} F_n +B_{r,n}\right)\right)\right|\\
&=&\sup_{\phi\in\mathcal{E}_M} \left| E\left(\phi^{(k)}(F_n )\sum_{r=1}^{k+1}\frac{P^{(r)}(0)}{r! 2^{r-1}}\Big(\Gamma_{r-1}(F_n )- E(\Gamma_{r-1}(F_n ))\Big)\right)\right|\\
&=&\sup_{\phi\in \mathcal{F}_M} \left| E\left(\phi(F_n )\sum_{r=1}^{k+1}\frac{P^{(r)}(0)}{r! 2^{r-1}}\Big(\Gamma_{r-1}(F_n )- E(\Gamma_{r-1}(F_n ))\Big)\right)\right|\\
&&\xrightarrow[n\to\infty]{{\rm (Step 3)}}~0.
\end{eqnarray*}
To achieve the proof, we notice that
\begin{eqnarray*}
&& E\left(\left| E\left(\sum_{r=1}^{k+1}\frac{P^{(r)}(0)}{r! 2^{r-1}}\Big(\Gamma_{r-1}(F_n )- E(\Gamma_{r-1}(F_n ))\Big)\Big|F_n \right)\right|\right)\\
&=&\sup_{\|\phi\|_\infty\le 1}\left| E\left(\phi(F_n )\sum_{r=1}^{k+1}\frac{P^{(r)}(0)}{r! 2^{r-1}}\Big(\Gamma_{r-1}(F_n )- E(\Gamma_{r-1}(F_n ))\Big)\right)\right|\\
&\le&\sup_{\phi\in \mathcal{F}_M} \left| E\left(\phi(F_n )\sum_{r=1}^{k+1}\frac{P^{(r)}(0)}{r! 2^{r-1}}\Big(\Gamma_{r-1}(F_n )- E(\Gamma_{r-1}(F_n ))\Big)\right)\right|\\
&+& E\left(\mathbf{1}_{\{|F_n |>M\}}\left|\sum_{r=1}^{k+1}\frac{P^{(r)}(0)}{r! 2^{r-1}}\Big(\Gamma_{r-1}(F_n )- E(\Gamma_{r-1}(F_n ))\Big)\right|\right)\\
&\le&\sup_{\phi\in \mathcal{F}_M} \left| E\left(\phi(F_n )\sum_{r=1}^{k+1}\frac{P^{(r)}(0)}{r! 2^{r-1}}\Big(\Gamma_{r-1}(F_n )- E(\Gamma_{r-1}(F_n ))\Big)\right)\right|\\
&+&\sqrt{\P(|F_n |>M)}\times\\
&&\quad\quad \times \sup_n\sqrt{ E \Bigg( \sum_{r=1}^{k+1} \frac{P^{(r)}(0)}{r! \ 2^{r-1}} \Big( \Gamma_{r-1}(F_n ) -   E(\Gamma_{r-1}(F_n )) \Big)\Bigg)^2}\\
&& {\xrightarrow[M\to\infty]{}~0.}
\end{eqnarray*}

\qed
\begin{rmk}[On Theorem \ref{t:noup}]{\rm \label{r:r} As anticipated, Theorem \ref{main-general-thm} allows one to deduce an alternate proof of the implication {\bf (ii)} $\rightarrow$ {\bf (i)} in Theorem \ref{t:noup}. Indeed, if 
Assumption {\bf (ii)} in Theorem \ref{t:noup} is verified, one can apply \eqref{e:4} to deduce that
$$
E \Bigg( \sum_{r=1}^{\text{deg}(P)} \frac{P^{(r)}(0)}{r! \ 2^{r-1}} \Big( \Gamma_{r-1}(F_n) -   E(\Gamma_{r-1}(F_n)) \Big) \Bigg)^2\to 0,
$$
and the conclusion follows immediately from Theorem \ref{main-general-thm}, as well as a standard application of Jensen's inequality. }
\end{rmk}

\section{Example: two eigenvalues}

We will now illustrate the main findings of the present paper by considering the case of a target random variable of the type $F_\infty =I_2(f_\infty)$, where the Hilbert-Schmidt operator $A_{f_\infty}$ associated the kernel $f_\infty$ has only two non-zero eigenvalues $\alpha_1 \neq \alpha_2$, thus implying that
\begin{equation}\label{target3}
F_\infty=\alpha_1 \left( N^{2}_{1} -1\right) + \alpha_2 \left( N^{2}_{2} -1\right)
\end{equation}
where $N_1$ and $N_2$ are independent $\mathcal{N}(0,1)$ (see Proposition \ref{second-property}).

\begin{thm}\label{q-chaos}
Assume that $F_\infty= I_2(f_\infty)$ given by \eqref{target3}. Let $q\ge2$ and $\{ F_n \}_{n \ge 1}= \{I_q(f_n)\}_{n \ge 1}$ be a sequence of multiple Wiener integrals of order $q$ such that 
\begin{equation*}
\lim_{n \to \infty} E(F_n^2)= 2 \lim_{n \to \infty} \Vert f_n\Vert_{\HH^{\otimes q}}^2 =1.
\end{equation*}
Assume that, as $n$ tends to infinity, we have
\begin{itemize} 
 \item[\bf (a)]$ \langle f_n{\tilde{\otimes}}_{\frac{q}{2}}f_n, f_n \rangle_{\HH^{\otimes q}} \to 0, \quad \text{when} \ q \ \text{is even}$, and
\item[\bf (b)] the following three asymptotic conditions {\bf (b1)}--{\bf (b3}) take place:
\begin{itemize}
\item[\bf (b1)]
\begin{eqnarray*}
 && \Bigg \Vert  \mathop{\sum_{r=1}^{q} \sum_{s=1}^{(2q-2r) \wedge q}}_{r+s=q} \frac{1}{4}q^2 (r-1)!(s-1)!{q-1 \choose r-1}^2   \times \\
&& \times {q-1 \choose s-1}  {2q-2r-1  \choose s-1}  \left( f_n{\tilde{\otimes}}_{r}f_n \right){\tilde{\otimes}}_{s}f_n \\ 
 &&- \frac{\alpha_1 + \alpha_2}{2} q (\frac{q}{2} -1)! {q-1 \choose \frac{q}{2}-1} f_n{\tilde{\otimes}}_{\frac{q}{2}}f_n  + \alpha_1 \alpha_2  f_n 
 \Bigg\Vert_{\HH^{\otimes q}}^2\!\!\!\!\!\! \to 0
 \end{eqnarray*}
(when $q$ is not even or $\alpha_1 = - \alpha_2$, then the term in the middle -- involving the contraction of order $\frac{q}{2}$ -- is removed automatically).

 \item[\bf (b2)] for all $2 \le k \le 2q -2$, we have 
 
 \begin{eqnarray*}
 && \Bigg \Vert \mathop{ \mathop{\mathop{\sum_{r=1}^{q} \sum_{s=1}^{(2q-2r) \wedge q}}_{(r,s)\neq(\frac{q}{2},q)}}_{r+s\neq q}}_{3q-2(r+s)=k}  (r-1)!(s-1)!{q-1 \choose r-1}^2\times \\
  &&\times  {q-1 \choose s-1} 
   {2q-2r-1 \choose s-1}  \left( f_n{\tilde{\otimes}}_{r}f_n \right){\tilde{\otimes}}_{s}f_n  \\
  && - \frac{\alpha_1 + \alpha_2}{2} \mathop{\sum_{ \frac{q}{2} \neq r=1 }^{q-1}}_{2q -2r=k} q (r -1)! {q -1 \choose r-1}^2 f_n{\tilde{\otimes}}_{r}f_n  \Bigg\Vert_{\HH ^{\otimes k}}^2 \to 0.
\end{eqnarray*}

\item[\bf (b3)] for all $ 2q -1\le  k \le 3q-4$, we have
\begin{equation*}
\begin{split}
 \Bigg \Vert \mathop{ \mathop{\mathop{\sum_{r=1}^{q} \sum_{s=1}^{(2q-2r) \wedge q}}_{(r,s)\neq(\frac{q}{2},q)}}_{r+s\neq q}}_{3q-2(r+s)=k}  (r-1)!(s-1)!{q-1 \choose r-1}^2  {q-1 \choose s-1} & {2q-2r-1  \choose s-1} \\
 \times \left( f_n{\tilde{\otimes}}_{r}f_n \right){\tilde{\otimes}}_{s}f_n  \Bigg\Vert_{\HH ^{\otimes k}}^2 \to 0.
\end{split}
 \end{equation*}
 \end{itemize}
 \end{itemize}
Then,
\begin{equation*}
F_n \stackrel{{\rm law}}{\to} F_\infty.
\end{equation*}
\end{thm}

\begin{proof}
In this case, a simple application of Jensen's inequality shows that the second moment of the quantity appearing on the left-hand side of Point $2$ of Theorem \ref{main-general-thm}-$\bf(ii)$ is bounded from 
above by 
\begin{equation*}
\begin{split}
 E\Bigg( \frac{\Gamma_2(F_n) \!-\!  E(\Gamma_2(F_n))}{4} &\! -\! \frac{\alpha_1 + \alpha_2}{2} \Big( \Gamma_1(F_n)\! -\!  E(\Gamma_1(F_n)) \Big) \!+\! \alpha_1 \alpha_2 F_n\!\Bigg)^2.
\end{split}
\end{equation*}
The claim follows immediately from Definition \ref{Def : Gamma}, orthogonality of multiple Wiener integrals, Theorem \ref{main-general-thm} and the fact that when $q$ is even
\begin{equation*}
\kappa_3(F_n)= 2  E(\Gamma_2(F_n))=  2 q q! (\frac{q}{2}-1)! {q-1 \choose \frac{q}{2}-1}^2 \langle f_n{\tilde{\otimes}}_{\frac{q}{2}}f_n, f_n \rangle_{\HH^{\otimes q}}.
\end{equation*}
\end{proof}

In the special case when $\alpha_1= -\alpha_2 =\frac{1}{2}$, the target random variable $F_\infty= I_2(f_\infty)$ in the limit takes the form
\begin{equation}\label{target2}
F_\infty=\frac{1}{2} \left( N^{2}_{1} -1\right) - \frac{1}{2} \left( N^{2}_{2} -1\right) \stackrel{{\rm law}}{=} N_1 N_2
\end{equation}
where $N_1$ and $N_2$ are independent $\mathcal{N}(0,1)$. If the elements $F_n$ of approximating sequence take the special form of multiple Wiener integrals of a fixed order, then we have the following result. One should notice that, in this special case, the result stated below can be alternatively deduced from the findings contained in \cite{et}. For a \textit{free} counterpart of the next result, see \cite[Theorem 1.1]{d-n}.

\begin{coro}\label{q-chaos}
Assume that $F_\infty= I_2(f_\infty)$ given by (\ref{target2}). Let $q\ge2$ and $\{ F_n \}_{n \ge 1}= \{I_q(f_n)\}_{n \ge 1}$ be a sequence of multiple Wiener integrals of order $q$ such that 
\begin{equation*}
\lim_{n \to \infty} E(F_n^2)= 2 \lim_{n \to \infty} \Vert f_n\Vert_{\HH^{\otimes q}}^2 =1.
\end{equation*}
Assume that, as $n$ tends to infinity, we have
\begin{itemize} 
 \item[\bf (a)]$ \langle f_n{\tilde{\otimes}}_{\frac{q}{2}}f_n, f_n \rangle_{\HH^{\otimes q}} \to 0, \quad \text{when} \ q \ \text{is even}$,
\item[\bf (b)] and moreover
\begin{itemize}
\item[\bf (b1)]
\begin{eqnarray*}
&& \Bigg \Vert \mathop{\sum_{r=1}^{q} \sum_{s=1}^{(2q-2r) \wedge q}}_{r+s=q} q^2 (r-1)!(s-1)!{q-1 \choose r-1}^2 \times \\
&& \times {q-1 \choose s-1}  {2q-2r-1  \choose s-1}
 \left( f_n{\tilde{\otimes}}_{r}f_n \right){\tilde{\otimes}}_{s}f_n  - f_n \Bigg\Vert_{\HH^{\otimes q}}^2 \!\!\!\!\!\! \to 0.
\end{eqnarray*}
\item[\bf (b2)] for all $2 \le k \le 3q-4$, we have
\begin{eqnarray*}
&& \Bigg \Vert \mathop{ \mathop{\mathop{\sum_{r=1}^{q} \sum_{s=1}^{(2q-2r) \wedge q}}_{(r,s)\neq(\frac{q}{2},q)}}_{r+s\neq q}}_{3q-2(r+s)=k}  (r-1)!(s-1)!{q-1 \choose r-1}^2  {q-1 \choose s-1} \times \\
&& \times {2q-2r-1  \choose s-1}  \times \left( f_n{\tilde{\otimes}}_{r}f_n \right){\tilde{\otimes}}_{s}f_n  \Bigg\Vert_{\HH ^{\otimes k}}^2 \to 0.
 \end{eqnarray*}
 \end{itemize}
 \end{itemize}
Then,
\begin{equation*}
F_n \stackrel{{\rm law}}{\to} F_\infty.
\end{equation*}
\end{coro}

\begin{proof}
In this special case, similarly, a simple application of Jensen's inequality shows that the second moment of the quantity appearing on the left-hand side of Point $2$ of Theorem 
\ref{main-general-thm}-$\bf(ii)$ is bounded from 
above by
\begin{equation*}
 E\Bigg( \frac{\Gamma_2(F)}{4} -  \frac{ E(\Gamma_2(F))}{4} - \frac{1}{4}F\Bigg)^2=:  E( A_1 +A_2)^2
\end{equation*}
where
\begin{eqnarray*}\label{A_1}
&& A_1= \frac{1}{4} I_q\Bigg( \mathop{\sum_{r=1}^{q} \sum_{s=1}^{(2q-2r) \wedge q}}_{r+s=q} q^2 (r-1)!(s-1)!{q-1 \choose r-1}^2  \times \\
&& \times {q-1 \choose s-1}  {2q-2r-1  \choose s-1} 
  \times \left( f_n{\tilde{\otimes}}_{r}f_n \right){\tilde{\otimes}}_{s}f_n  - f_n \Bigg)\\
 \end{eqnarray*}
and

\begin{eqnarray*}\label{A_2}
&& A_2 = \frac{1}{4}\mathop{\mathop{\sum_{r=1}^{q} \sum_{s=1}^{(2q-2r) \wedge q}}_{(r,s)\neq(\frac{q}{2},q)}}_{r+s\neq q} q^2 (r-1)!(s-1)!{q-1 \choose r-1}^2  \times \\
&& \times {q-1 \choose s-1}  {2q-2r-1  \choose s-1}
 \times I_{3q-2r-2s}\Big( (f_n{\tilde{\otimes}}_{r}f_n){\tilde{\otimes}}_{s}f_n\Big)\\
&& = \frac{1}{4}\mathop{\sum_{k=1}^{3q-4}}_{k \neq q} I_{k}\Bigg( \mathop{\mathop{\sum_{r=1}^{q} \sum_{s=1}^{(2q-2r) \wedge q} }_{(r,s)\neq(\frac{q}{2},q)}}_{3q-2(r+s)=k} q^2 (r-1)!(s-1)!{q-1 \choose r-1}^2 \times \\ 
&& \times {q-1 \choose s-1}  {2q-2r-1  \choose s-1}
 \times (f_n{\tilde{\otimes}}_{r}f_n){\tilde{\otimes}}_{s}f_n\Bigg)
\end{eqnarray*}
Now the claim follows immediately by orthogonality of multiple Wiener integrals, Theorem \ref{main-general-thm} and the fact that when $q$ is even
\begin{equation*}
\kappa_3(F_n)= 2  E(\Gamma_2(F_n))=  2 q q! (\frac{q}{2}-1)! {q-1 \choose \frac{q}{2}-1}^2 \langle f_n{\tilde{\otimes}}_{\frac{q}{2}}f_n, f_n \rangle_{\HH^{\otimes q}}.
\end{equation*}

\end{proof}

\begin{rmk}{\rm
Notice that when $q$ is odd, the assumption $\bf (a)$ of Theorem \ref{q-chaos} and restriction $(r,s)\neq(\frac{q}{2},q)$ in the sums of $\bf (b2)$ can be removed. In other words, it is known that for any multiple 
Wiener integral $F=I_q(f)$ of odd order, we have $\kappa_3(F)= 2  E(\Gamma_2(F))=0$.}
\end{rmk}

\begin{ex}{ \rm
Let $q \ge 2$ be an even integer. Consider two sequences $\{G_n\}_{n\ge 1}=\{I_q(g_n)\}_{n\ge 1}$ and $\{H_n\}_{n\ge 1}=\{I_q(h_n)\}_{n\ge 1}$ of multiple Wiener integrals of order $q$ where 
$g_n, h_n \in \HH^{\odot q}$ for $n \ge 1$. 
We assume that as $n$ tends to infinity we have 
\begin{itemize}
 \item[(a)] $G_n \stackrel{\text{law}}{\to} G_\infty \stackrel{\text{law}}{=} \frac{1}{2} (N^2-1)$.
 \item[(b)] $H_n \stackrel{\text{law}}{\to} H_\infty \stackrel{\text{law}}{=}\frac{1}{2} (N^2-1)$.
\item[(c)] $\text{Cov}(G_n^2,H_n^2) \to 0$.
 \end{itemize}
We consider the sequence $\{ F_n\}_{n \ge 1}$, where
\begin{equation*}
F_n= I_q(f_n):= G_n -H_n =I_q(g_n-h_n), \quad n\ge 1.
\end{equation*}
Then \cite[Theorem 4.5.]{n-ro} implies that as $n$ tends to infinity, we have 
\begin{equation*}
(G_n,H_n) \stackrel{{\rm law}}{\to} (G_\infty,H_\infty).
\end{equation*}
Hence, in particular we obtain that $F_n \stackrel{\text{law}}{\to} F_\infty$ where $F_\infty$ is given by (\ref{target2}). We can also justify the later convergence with the help of our result, namely 
Theorem \ref{main-general-thm}. To this end, first notice that relation $(3.26)$ of \cite{n-ro} implies that 
\begin{equation}\label{covariance-inequality}
\text{Cov}(G_n^2,H_n^2) \ge  E(G_nH_n).
\end{equation}
Therefore, using assumption $(c)$ we obtain that $\kappa_2(F_n) \to \kappa_2(F_\infty)=1$. Moreover, according to \cite[Theorem 1.2]{n-pe-2} point (iii), assumption (a) implies that for constant 
$c_q=4 \frac{{(\frac{q}{2})!}^3}{{q!}^2}$, we have $\Vert g_n{\tilde{\otimes}}_{\frac{q}{2}}g_n - c_q g_n\Vert_{\HH^{\otimes q}} \to 0$. Hence
\begin{eqnarray*}
&&\left \vert  E(G_n^2 H_n) \right \vert\\ 
&&= \left \vert\langle g_n{\tilde{\otimes}}_{\frac{q}{2}}g_n,h_n\rangle_{\HH^{\otimes q}} 
\right \vert \\
&&= \left \vert \langle g_n{\tilde{\otimes}}_{\frac{q}{2}}g_n - c_q g_n,h_n\rangle_{\HH^{\otimes q}} 
  + c_q \langle g_n,h_n\rangle_{\HH^{\otimes q}} \right \vert\\
&& \le \Vert g_n{\tilde{\otimes}}_{\frac{q}{2}}g_n - c_q g_n\Vert_{\HH^{\otimes q}}^2 \Vert h_n \Vert_{\HH^{\otimes q}}^2 +c_q q!  E(G_n H_n) \\
&& \to 0,
\end{eqnarray*}
by assumptions (a), (b), and $(\ref{covariance-inequality})$. In a similar way, one can see that $ E(G_n H_n^2) \to 0$. Hence

\begin{equation*}
\kappa_3(F_n)=\kappa_3(G_n)-3 E(G_n^2 H_n) +3  E(G_n H_n^2) - \kappa_3(H_n) \to 0.
\end{equation*}
To complete, notice that (see \cite[Theorem 3.1]{n-ro}) Assumption (c) tells us that $$\Vert g_n {\tilde{\otimes}}_{r} h_n \Vert_{\HH^{\otimes (2q-2r)}} \to 0, \quad \forall r=1, 2, \cdots, q.$$
Using this observation together with straightforward computations, one can see that for the sequence $\{F_n\}_{n \ge 1}$ we have
\begin{equation*}
 E\Bigg( \frac{\Gamma_2(F_n)}{4} -  \frac{ E(\Gamma_2(F_n))}{4} - \frac{1}{4}F_n \Bigg)^2 \to 0.
\end{equation*}

}
\end{ex}

\section{Appendix}
\begin{lemma}\label{lemma-appendix}
Let $\{a_k\}_{k \in \N}$ and $\{b_k\}_{k \in \N}$ be two sequences in $l^{1}(\N)$ such that for all $p \ge 1$ we have 

\begin{equation}\label{=sums}
\sum_{k\in \N} a_k^p = \sum_{k \in \N} b_k^p.
\end{equation}
Then, there exists a permutation $\pi$ on natural numbers $\N$ such that $a_k = b_{\pi(k)}$ for all $k \ge 1$.
\end{lemma}

\begin{proof}
 Let $\R[X]$ denote the ring of all polynomials over real line. Then, relation $(\ref{=sums})$ implies that for any polynomial $P \in \R[X]$, we have 
 
 \begin{equation}
  \sum_{k \in \N} a_k P(a_k) = \sum_{k \in \N} b_k P(b_k).
 \end{equation}
Let $M:= \max \{ \Vert a \Vert_{l^1(\N)}, \Vert b \Vert_{l^1(\N)} \} < \infty$. Then by a density argument, for any continuous function $\varphi \in C([-M,M])$, we obtain 
\begin{equation}
\sum_{k \in \N} a_k \varphi(a_k) = \sum_{k \in \N} b_k \varphi(b_k).
\end{equation}
For any $i\in\N$, we can now choose a continuous function $\varphi$ such that $\varphi(a_i)=1$ and $\varphi =0$ on the set $\{a_j| a_j \neq a_i\}\cup \{b_j| b_j \neq a_i\}$. This implies that, 
for some integer $k_i$, we have $a_i=b_{k_i}$. It is now sufficient to take $\pi(i)=k_i$.
\end{proof}

\medskip

\noindent{\bf Acknoledgements}.  We thank P. Eichelsbacher and Ch. Th\"ale for discussing with us, at a preliminary stage, the results contained in reference \cite{et}. EA \& GP were partially supported by the Grant F1R-MTH-PUL-12PAMP
(PAMPAS) from Luxembourg University.

\end{document}